\documentclass[sn-mathphys-num]{sn-jnl}

\usepackage{graphicx}%
\usepackage{multirow}%
\usepackage{amsmath,amssymb,amsfonts}%
\usepackage{amsthm}%
\usepackage{mathrsfs}%
\usepackage[title]{appendix}%
\usepackage{xcolor}%
\usepackage{textcomp}%
\usepackage{manyfoot}%
\usepackage{booktabs}%
\usepackage{algorithm}%
\usepackage{algorithmicx}%
\usepackage{algpseudocode}%
\usepackage{listings}%
\usepackage[all]{xy}%

\newtheorem*{thm*}{Main Theorem}
\newtheorem{theorem}{Theorem}[section]
\newtheorem{corollary}[theorem]{Corollary}
\newtheorem*{corollary*}{Corollary}
\newtheorem{lemma}[theorem]{Lemma}
\newtheorem*{lemma*}{Lemma}

\newtheorem*{proposition*}{Proposition}
\newtheorem{remark}[theorem]{Remark}
\newtheorem*{remark*}{Remark}
\newtheorem{definition}[theorem]{Definition}
\newtheorem*{definition*}{Definition}
\newtheorem{example}[theorem]{Example}
\newtheorem*{example*}{Example}

\newcommand{\Ext}{\mbox{\rm Ext}}
\newcommand{\Hom}{\mbox{\rm Hom}}
\newcommand{\gen}{\mbox{\rm gen}}

\newcommand{\End}{\mbox{\rm End}}
\newcommand{\Ker}{\mbox{\rm Ker}}
\newcommand{\Coker}{\mbox{\rm Coker}}
\newcommand{\id}{\mbox{\rm id}}
\newcommand{\pd}{\mbox{\rm pd}}
\newcommand{\Mod}{\mbox{\rm Mod}}
\newcommand{\add}{\mbox{\rm add}}
\newcommand{\fd}{\mbox{\rm fd}}
\newcommand{\gd}{\mbox{\rm gl.dim}}

\newcommand{\M}{\rm Morph}

\def\mod{\mathop{\rm mod}\nolimits}

\begin{document}

\title[Article Title]{Constructions of Symmetrical Separable Equivalences and Their Applications}


\author[1]{\fnm{Juxiang} \sur{Sun}}\email{Sunjx8078@163.com}

\author*[2]{\fnm{Guoqiang} \sur{Zhao}}\email{gqzhao@hdu.edu.cn}


\affil[1]{\orgdiv{School of Mathematics and Statistics}, \orgname{Shangqiu Normal University}, \orgaddress{\street{}\city{Shangqiu}, \postcode{476000}, \state{Henan}, \country{China}}}

\affil*[2]{\orgdiv{School of Science}, \orgname{Hangzhou Dianzi University}, \orgaddress{\street{} \city{Hangzhou}, \postcode{310018}, \state{Zhejiang}, \country{China}}}


\abstract{Let $\Lambda$ and $\Gamma$ be symmetrically separably equivalent Artin algebras.
We prove that there exist symmetrical separable equivalences between
 certain endomorphism algebras of modules. As applications,
 we provide several methods to construct symmetrical separable equivalences from given ones and
 discuss when the rigidity dimension is an invariant under symmetrical separable equivalences.
 Moreover, we show that a symmetrical separable equivalence preserves the Frobenius-finite type,
 Auslander-type condition, the (strong) Nakayama conjecture, the Auslander-Gorenstein conjecture and so on.}




\maketitle

\section{Introduction}\label{sec1}

As a weakening of Morita equivalence,
the notion of the (symmetrical) separable equivalence was introduced
in \cite{K1, L1},
with many interesting applications (see \cite{B,BE,BT,K1,K2,L1,L2} for detail). The fundamental examples
are separable split Frobenius ring extensions and the stable equivalences of adjoint type \cite{Xi}.
It was proved in \cite{K2} that two rings $\Lambda$ and $\Gamma$ are symmetrically separably equivalent
if and only if they are linked by a biseparable Frobenius bimodule.
It has been known that two symmetrically separably equivalent Artin algebras
share many important properties of  homology and representation,
such as global dimension,
representation dimension, finitely generated cohomology, dominant dimension
and so on (see \cite{B,K1,L2} for detail).

The  first part of the present paper is devoted to providing several methods
to produce a new symmetrical separable equivalence from a given one.
Our main construction is the following result.

\vspace{0.2cm}

{\bf Theorem A} (Theorem 3.1)
 {\it Let $\Lambda$ and $\Gamma$ be symmetrically separably equivalent defined by $_{\Lambda}M_{\Gamma}$ and $_{\Gamma}N_{\Lambda}$.
 Let $U$ be a finitely generated $\Lambda$-module and $V$ a finitely generated $\Gamma$-module such that
 $N\otimes_{\Lambda}U \in \add_{\Gamma}V$ and $M\otimes_{\Gamma}V \in \add_{\Lambda}U$.
 Then there exists a symmetrical separable equivalence between  endomorphism algebras $\End_{\Lambda}U$  and $\End_{\Gamma}V$.}

 \vspace{0.2cm}

 As applications of Theorem A, we have the following constructions which show that there do exist algebras and modules satisfying all conditions in Theorem A.

 \vspace{0.2cm}
{\bf Theorem B} (Theorem 3.2 and Theorem 3.4) {\it Let $\Lambda$ and $\Gamma$ be symmetrically separably equivalent Artin algebras.

$\mathrm{(1)}$ Assume that $\Lambda$ and $\Gamma$ are representation-finite, then their Auslander algebras are also symmetrically separably equivalent.

$\mathrm{(2)}$ The Frobenius part of $\Lambda$ and that of $\Gamma$ are symmetrically separably equivalent.

}

\vspace{0.2cm}

As another consequence of Theorem A, we construct a new symmetrical separable equivalence from a given stable equivalence of adjoint type and a tilting module.

\vspace{0.2cm}

{\bf Theorem C} (Theorem 3.7)
{\it Let $\Lambda$ and $\Gamma$ be stably equivalent of adjoint type  defined by $_{\Lambda}M_{\Gamma}$ and $_{\Gamma}N_{\Lambda}$ and $m$ a positive integer.
If $T$ is a $m$-tilting $\Lambda$-module, then there exists a symmetrical separable equivalence between endomorphism algebras $\End_{\Lambda}T$ and $\End_{\Gamma}(N\otimes_{\Lambda}T)$.}

\vspace{0.2cm}

As applications of the theorems above, the invariant of the rigidity dimension,
the Frobenius-finite type and the tilting property of algebras under symmetrical separable equivalences are investigated.

The second part of this paper is to investigate invariant properties of the Auslander-type condition and we obtain the following result.

\vspace{0.2cm}

{\bf Theorem D} (Theorem 4.5) {\it Let $\Lambda$ and $\Gamma$ be left and right noetherian rings.
Suppose that there is a symmetrical separable equivalence between $\Lambda$ and $\Gamma$.
Then $\Lambda$ satisfies the Auslander-type condition  if and only if so does $\Gamma$.}

All proofs of the results above will be given in Section 3 and Section 4 after we recall some necessary notations and results in Section 2.

\section{Preliminaries}\label{sec2}

Throughout this paper,  all rings are left and right noetherian rings, and all modules are left modules unless stated otherwise.
For a ring $\Lambda$, we denote by $\Mod\Lambda$ (resp. $\mod \Lambda$) the category consisting of all (resp. finitely generated) left $\Lambda$-modules.
We write $\pd_{\Lambda}M$ (resp. $\id_{\Lambda}M$, $\fd_{\Lambda}M$) to denote the projective dimension (resp. the injective, flat dimension) of $M$,
and the global dimension of $\Lambda$ is denoted by gl.dim$\Lambda$.
Let $M, N$ be finitely generated $\Lambda$-modules, we denote by $\add M$ the subcategory of $\mod\Lambda$
consisting of all direct summands of finite copies of $M$,  and we set $N|M$, if $N$ is isomorphic to a direct summand of $M$.

\begin{definition}[\cite{K1, L1}]
Let $\Lambda$ and $\Gamma$ be two rings, and let bimodules $_{\Lambda}M_{\Gamma}$ and $_{\Gamma}N_{\Lambda}$
be finitely generated projective as one-sided modules.
$\Lambda$ and $\Gamma$ are said to be  separably equivalent induced by bimodules $_{\Lambda}M_{\Gamma}$ and  $_{\Gamma}N_{\Lambda}$,
if there exist bimodules  $_\Lambda X_\Lambda$ and $_\Gamma Y_\Gamma$ and bimodule isomorphisms
 $$_{\Lambda}M\otimes_{\Gamma}N_{\Lambda}\cong _{\Lambda}\Lambda_{\Lambda}\oplus _{\Lambda}X_\Lambda \qquad and
 \qquad_{\Gamma}N\otimes_{\Lambda}M_{\Gamma}\cong _{\Gamma}\Gamma_{\Gamma}\oplus _{\Gamma}Y_{\Gamma}.$$
 Furthermore, assume that $(M\otimes_{\Gamma}-, N\otimes_{\Lambda}-)$ and $(N\otimes_{\Lambda}-, M\otimes_{\Gamma}-)$ are adjoint pairs,
 then $\Lambda$ and $\Gamma$ are said to be symmetrically separably equivalent.
\end{definition}

\begin{remark}
Let $\Lambda$ and $\Gamma$ be symmetrically separably equivalent defined as Definition 2.1.

$\mathrm{(1)}$ If the bimodules $_\Lambda X_\Lambda$ and $_\Gamma Y_\Gamma$ are the zero modules,
then $\Lambda$ is Morita equivalent to $\Gamma$.

$\mathrm{(2)}$ If  $_\Lambda X_\Lambda$ and $_\Gamma Y_\Gamma$ are projective as bimodules,
then $\Lambda$ and $\Gamma$ are stably equivalent of adjoint type (see \rm{\cite{Xi}}) .
\end{remark}

\begin{example}
{\rm $\mathrm{(1)}$ (see \cite{K1}) For a ring $\Lambda$, $M_n(\Lambda)$ ( the matrix ring of $\Lambda$ degree in $n$ )
and $\Lambda$ are symmetrically separably equivalent;

$ \mathrm{(2)}$ (see \cite[Proposition 2]{P1}) Let $k$ be a field of characteristic $p>0$, and let $G$ be a finite group.
If $H$ is a Sylow $p$-subgroup of $G$,
then $kG$ and $kH$ are symmetrically  separably equivalent;

$\mathrm{(3)}$ (see \cite[Lemma 1]{BT}) Let $\Lambda$ be a ring and $G$ be a finite group.
If the order of $G$ is invertible in $\Lambda$, then the skew group ring $\Lambda G$ and  $\Lambda$ are symmetrically separably equivalent;

$\mathrm{(4)}$ (see \cite[Lemma 4.6]{SZ}) Let $\Lambda$ be a finite dimensional algebra over an algebraically closed field $k$,
and  $F$ be  a finite separable field extension of $k$, then $\Lambda$ and $\Lambda\otimes_{k}F$ are  symmetrically separably equivalent.

More examples of symmetrically separable equivalences can be found in \cite{P1}.}
\end{example}

For our purpose we also need the following fact.

\begin{lemma}[see \cite{K1,SZ}]
 Let $\Lambda$ and $\Gamma$ be symmetrically separably equivalent defined as Definition 2.1. Then

$\mathrm{(1)}$   $_{\Lambda}M_{\Gamma}$  and $_{\Gamma}N_{\Lambda}$ are projective generators as one-sided modules;

$\mathrm{(2)}$   $N\otimes_{\Lambda}-$ and $M\otimes_{\Gamma}-$ are exact functors and take projective modules to projective modules;

$\mathrm{(3)}$  There are natural isomorphisms
 $(M\otimes_{\Gamma}-)\circ(N\otimes_{\Lambda}-)\to \mathrm{Id}_{\Mod \Lambda}\oplus (X\otimes_{\Lambda}-)$ and
 $(N\otimes_{\Lambda}-)\circ(M\otimes_{\Gamma}-)\to \mathrm{Id}_{\Mod\Gamma}\oplus (Y\otimes_{\Gamma}-).$

$\mathrm{(4)}$ $\gd\Lambda=\gd\Gamma$.
\end{lemma}

Let $\Lambda$ be an Artin algebra, and let $\mathcal{C}$ be a subcategory of $\mod\Lambda$.
The morphism category \cite{A}, denoted by Morph$(\mathcal{C})$, is defined as follows.
The objects in  Morph$(\mathcal{C})$ are all the morphisms $ f: C_1\to C_0$ in $\mathcal{C}$, denoted by $(C_1, C_0, f)$.
The morphism from  an object $ (C_1, C_0, f)$ to another object $ (C'_1, C'_0, f')$ is a pair $(\alpha_1,\alpha_0)$,
where $\alpha_i: C_i\to C'_i$ for $i=0,1$, such that $\alpha_0f=f'\alpha_1$.
For two objects $f: C_1\to C_0$ and $ f': C'_1\to C'_0$,
let $\mathcal{R_C}(f, f')=\{(\alpha_1,\alpha_0):  (C_1, C_0, f)\to  (C'_1, C'_0, f')$ such that there is some $\beta: C_0\to C'_1$ such that $\alpha_0=f'\beta \}$.
It is not hard to check that $\mathcal{R_C}$ gives an additive equivalence relation on $\M(\mathcal{C})$.
The following lemma is due to \cite[Lemma 2.4]{LX}.

\begin{lemma}
Let $\Lambda$ and $\Gamma$ be Artin algebras and $U\in\mod \Lambda$.
Then there is an equivalence $\mathrm{H}_A: \M(\add_\Lambda U)/\mathcal{R}_U\to \mod A$,
where $A=\End_{\Lambda}U$ and $\mathcal{R}_U=\mathcal{R}_{\add_{\Lambda}U}$,
given by $\mathrm{H}_A((U_1,U_0,f))=\Coker (\Hom_{\Lambda}(U,f))$ for any object $(U_1,U_0,f)$ in $\M(\add_\Lambda U)$,
and morphism $\mathrm{H}_A(\alpha_1,\alpha_0): \mathrm{H}_A((U_1,U_0,f))\to\mathrm{ H}_A ((U'_1,U'_0,f'))$
to be the unique morphism which makes the diagram
$$\xymatrix{
  \Hom_{\Lambda}(U,U_1) \ar[d]_{(U,\alpha_1)} \ar[r]^{(U,f)} &  \Hom_{\Lambda}(U, U_0) \ar[d]_{(U,\alpha_0)} \ar[r] & \mathrm{H}_A((U_1,U_0,f))\ar[d] \ar[r] &0 \\
  \Hom_{\Lambda}(U,U'_1) \ar[r]^{(U, f')} &  \Hom_{\Lambda}(U, U'_0) \ar[r] & \mathrm{H}_A((U'_1,U'_0,f')) \ar[r] & 0  }$$
commute, for any morphism $(\alpha_1, \alpha_0)$ from an object $(U_1,U_0,f)$ to anther object $(U'_1,U'_0,f')$.
\end{lemma}

Let $K$ be a finitely generated $\Lambda$-module.
Recall that an exact sequence
$$0\to K \to C^0 \to C^1 \to \cdots \to C^n\to 0 $$
in $\mod \Lambda$ with all $C^i \in \mathcal{C}$ is called an  {\it $n$-$\mathcal{C}$-coresolution} of $K$.
An  {\it $n$-$\mathcal{C}$-coresolution} of $K$ is said to be {\it proper},
if it is also $\Hom_{\Lambda}(-, \mathcal{C})$-exact.
Dual to Lemma 4.2 in \cite{HS}, we have the following lemma.

\begin{lemma}
Let $\Lambda$ be an Artin algebra, let $X,X_1$ and $T$ be $\Lambda$-modules such that  $X_1 \in \add X$.
if $X$ has a proper $n$-$\add T $-coresolution:
$0\rightarrow X \stackrel{ f_{0}}{\longrightarrow} T_{0}\rightarrow\cdots \rightarrow T_{n-1} \stackrel{ f_{n}}{\longrightarrow}T_{n} \rightarrow 0$, then $X_{1}$ also has a proper $n$-$\add T $-coresolution:
$$  0\rightarrow X_1  \stackrel{ f'_{0}}{\longrightarrow} T'_{0}\rightarrow\cdots \rightarrow T'_{n-1} \stackrel{ f'_{n-1}}{\longrightarrow}T'_n \rightarrow 0$$
\end{lemma}

Let $\Lambda$ be a left and right noetherian ring and let
$$0\to \Lambda \to I^0(\Lambda)\to I^1(\Lambda)\to I^2(\Lambda)\to \cdots \to I^n(\Lambda)\to \cdots $$
be the minimal injective coresolution of $_{\Lambda}\Lambda$.
Following \cite{T},  {\it the dominant dimension} of $\Lambda$, denoted by dom.dim$\Lambda$, is defined by the minimal number $n$ such that
$I^n(\Lambda)$ is not projective. If $I^i(\Lambda)$ are injective-projective for all $i\geq 0$, then we have  dom.dim$\Lambda=\infty.$
For a positive integer $n$, recall from \cite{FGR} that $\Lambda$ is said to be an {\it $n$-Gorenstein ring},
 if $\fd_{\Lambda}I^i(\Lambda)\leq i$ for any $0\leq i\leq n-1,$ and $\Lambda$ is said to satisfy the {\it Auslander condition},
 if $\Lambda$ is $n$-Gorenstein ring for all $n$.  As a nontrivial generalization of the notion of the Auslander condition,
 Huang and Iyama in \cite{HI} introduced the notion of the Auslander-type condition as follows.
 For any $n,k\geq 0$, $\Lambda$ is said to be $G_n(k)$, if $\fd_{\Lambda}I^i(\Lambda)\leq k+i$,
  for all $0\leq i\leq n-1.$ It is easy to see that $\Lambda$ is an $n$-Gorenstein ring if and only if $\Lambda$ is $G_n(0)$.
The Auslander-type conditions play a significant role in representation theory of algebras,
homological algebras and non-commutative algebraic geometry (see \cite{Hu2, HI, IS,I1,I2} for detail).


Finally, let us list some homological conjectures, which are still open.

\vspace{0.1cm}

{\bf Auslander Gorenstein Conjecture:} An Artin algebra satisfying the Auslander condition  is a Gorenstein algebra (that is, the left
and right self-injective dimensions of $\Lambda$ are finite).

\vspace{0.1cm}

{\bf Nakayama Conjecture:} An Artin algebra $\Lambda$ with  dom.dim$\Lambda=\infty$ is self-injective.

\vspace{0.1cm}

{\bf Gorenstein Symmetric Conjecture:} Let $\Lambda$ be an Artin algebra, then $\id_\Lambda\Lambda=\id \Lambda_{\Lambda}$.

\vspace{0.1cm}

{\bf Strong  Nakayama Conjecture:} Any finitely generated module $M$ over an Artin algebra $\Lambda$ with  $\Ext_{\Lambda}^{\geq 0}(M, \Lambda)=0$ is zero.

\section{Construction of symmetrical separable equivalences}\label{sec3}



In this section, all rings are Artin $R$-algebras, where $R$ is a commutative Artin ring,
 and all modules are finitely generated left modules.
 We assume that $\Lambda$ and $\Gamma$ are symmetrically separably equivalent induced
 by bimodules $_\Lambda M_\Gamma$ and $_\Gamma N_\Lambda$.
 That is, there exist bimodules $_{\Lambda}X_{\Lambda}$ and $_{\Gamma}Y_{\Gamma}$,  and bimodule isomorphisms:
 $_{\Lambda}M\otimes_{\Gamma}N_{\Lambda}\cong  {_\Lambda}\Lambda_{\Lambda}\oplus {_{\Lambda}X_{\Lambda}}$,
and $_{\Gamma}N\otimes_{\Lambda}M_{\Gamma}\cong {_\Gamma}\Gamma_{\Gamma}\oplus {_\Gamma}Y_{\Gamma}$.
We define functors  $\mathrm{T}_N=N\otimes_{\Lambda}-: \Mod \Lambda\to \Mod \Gamma$ and
$\mathrm{T}_M=M\otimes_{\Gamma}-:  \Mod \Gamma\to \Mod \Lambda$,
which are exact functors such that $(\mathrm{T}_M, \mathrm{T}_N)$ and
$(\mathrm{T}_N, \mathrm{T}_M)$ are adjoint pairs by Lemma 2.4(2) and assumption.

\begin{theorem}
 Let $\Lambda$ and $\Gamma$ be symmetrically separably equivalent defined by $_{\Lambda}M_{\Gamma}$ and $_{\Gamma}N_{\Lambda}$.
 Let $U$ be a $\Lambda$-module and $V$ a $\Gamma$-module such that $N\otimes_{\Lambda}U \in \add_{\Gamma}V$
 and $M\otimes_{\Gamma}V \in \add_{\Lambda}U$. Then there exists a symmetrical separable equivalence between  endomorphism algebras $\End_{\Lambda}U$  and $\End_{\Gamma}V$.
\end{theorem}

\begin{proof}

We provide a proof by using some ideas in \cite[Theorem 1.1]{LX}.

Let $A=\End_{\Lambda}U$ and $B=\End_{\Gamma}V$.
Clearly, $A$ and $B$ are also Artin algebras. By Lemma 2.5, there exist
 two equivalences of categories
$$ \mathrm{H}_A: \M(\add_\Lambda U)/\mathcal{R}_U\to \mod A \quad and \quad
\mathrm{H}_B: \M(\add_\Gamma V)/\mathcal{R}_{V}\to \mod B.$$
Since  $N\otimes_{\Lambda}U \in \add_{\Gamma}V$,  we define a functor:
$$\widetilde{\mathrm{T}_N}:  \M(\add_\Lambda U)/\mathcal{R}_U\to \M(\add_\Gamma V)/\mathcal{R}_{V},$$
given by
$$\widetilde{\mathrm{T}_N}((U_1, U_0,f))=(\mathrm{T}_N(U_1),\mathrm{T}_N(U_0), \mathrm{T}_N(f)), $$
for any object $(U_1, U_0,f)$ in $\M(\add_\Lambda U)/\mathcal{R}_U$. And for a morphism $(\alpha_1,\alpha_0)+\mathcal{R}_U(f,f')$ from an object $(U_1, U_0, f)$ to another object $(U'_1, U'_0, f')$,
we define
$$\widetilde{\mathrm{T}_N}((\alpha_1,\alpha_0)+\mathcal{R}_U(f,f'))
=(\mathrm{T}_N(\alpha_1),\mathrm{T}_N(\alpha_0))+\mathcal{R}_{V}(\mathrm{T}_N(f),\mathrm{T}_N(f')).$$
Noting that $\mathrm{T}_N(\mathcal{R}_U(f,f'))\subseteq \mathcal{R}_{V}(\mathrm{T}_N(f),\mathrm{T}_N(f')),$
it is not hard to check that $\widetilde{\mathrm{T}_N}$ is well defined.
Similarly,  we have a functor $\widetilde{\mathrm{T}_M}:\M(\add_\Gamma V)/\mathcal{R}_{V}
\to \M(\add_\Lambda U)/\mathcal{R}_U$, because of $M\otimes_{\Gamma}V \in \add_{\Lambda}U$.

We denote by $\mathbb{F}$ the composition:
 $$\mod A \stackrel{\mathrm{H}^{-1}_A}\longrightarrow \M (\add_{\Lambda}U)/\mathcal{R}_U \stackrel{\widetilde{\mathrm{T}_N}}
 \longrightarrow\M (\add_{\Gamma} V)/\mathcal{R}_{V} \stackrel{\mathrm{H}_B}\longrightarrow \mod B, $$
and by $\mathbb{ G}$ the composition:
 $$\mod B \stackrel{\mathrm{H}^{-1}_B}\longrightarrow\M (\add_{\Gamma}V)/\mathcal{R}_{V}
 \stackrel{\widetilde{\mathrm{T}_M}} \longrightarrow \M (\add_{\Lambda}U)/\mathcal{R}_U \stackrel{H_A}\longrightarrow \mod A, $$
where $\mathrm{H}_A^{-1}$ ({\it resp.} $\mathrm{H}_B^{-1}$) is the inverse of $\mathrm{H}_A$ ({\it resp.} $\mathrm{H}_B$).
Clearly, $\mathbb{ F}$ and $\mathbb{ G}$ are additive functors.

 We divide the following proof in four steps.

{\bf Step 1}  We  show that $\mathbb{F}$ and $\mathbb{G}$ are exact functors.

Let $0\to D_1\to D_2\to D_3\to 0$ be a short exact sequence in $\mod A$.
Due to the definition of the equivalence $\mathrm{H}_A: \M(\add_{\Lambda}U)/\mathcal{R}_U \to \mod A$,
one has $H_A^{-1}(D_1)=(U_1, U_0, f)$ and $H_A^{-1}(D_3)=(U'_1, U'_0, f_3)$,
where $U_i,U'_i\in\add_{\Lambda}U$, for $i=0,1$. Then there are a projective presentation of $D_1$:
$$\Hom_{\Lambda}(U, U_1) \stackrel{(U,f_1)}\to \Hom_{\Lambda}(U, U_0) \to D_1\to 0$$
and that of $D_3$:
$$\Hom_{\Lambda}(U, U'_1) \stackrel{(U,f_3)}\to \Hom_{\Lambda}(U, U'_0) \to D_3\to 0.$$

Note that there are isomorphisms $\Hom_{\Lambda}(U, U_i)\oplus\Hom_{\Lambda}(U, U'_i) \cong \Hom_{\Lambda}(U, U_i\oplus U'_i)$ for $i=0,1$.
By Horseshoe lemma and Yoneda\rq s lemma, there exist $\alpha_i: U_i\to U_i\oplus U'_i$ and $\beta_i: U_i\oplus U'_i\to U'_i$ for $i=0,1$ and
$f_2: U_1\oplus U'_1\to U_0\oplus U'_0$ such that the following exact diagram is commutative

$$\xymatrix{
 \varepsilon_1\quad 0 \ar[r] & \Hom_{\Lambda}(U,U_1) \ar[d]_{(U,f_1)} \ar[r]^{(U,\alpha_1)} & \Hom_{\Lambda}(U,U_1\oplus U'_1) \ar[d]_{(U,f_2)} \ar[r]^{(U,\beta_1)} & \Hom_{\Lambda}(U,U'_1)\ar[d]_{(U,f_3)} \ar[r] & 0 \\
 \varepsilon_0\quad 0 \ar[r] & \Hom_{\Lambda}(U,U_0) \ar[d]\ar[r]^{(U,\alpha_0)} & \Hom_{\Lambda}(U,U_0\oplus U'_0)\ar[d] \ar[r]^{(U,\beta_0)} & \Hom_{\Lambda}(U,U'_1) \ar[d] \ar[r] & 0\qquad(3.1)\\
  0 \ar[r]  & D_1  \ar[d]\ar[r]  & D_2  \ar[d]\ar[r] & D_3  \ar[d]\ar[r]& 0 \\
    & 0   & 0  & 0 &    }$$
where $\varepsilon_0$ and $\varepsilon_1$ are
canonical split exact sequences.

Let $K_i=\Ker f_i$, for $i=1,2,3$. By Yoneda\rq s lemma and Snake lemma, there exists an exact commutative digram in $\mod \Lambda$ as follows
 $$\xymatrix{
   & 0 \ar[d]  & 0 \ar[d]  & 0 \ar[d]\\
  0\ar[r]& K_1 \ar[d] \ar[r]^{\alpha_2} & K_2\ar[d] \ar[r]^{\beta_2}& K_3 \ar[d] \\
  \varepsilon_3 \quad 0  \ar[r] & U_1 \ar[d]_{f_1} \ar[r]^{\alpha_1} & U_1\oplus U'_1 \ar[d]_{f_2} \ar[r]^{\beta_1} & U'_1 \ar[d]_{f_3} \ar[r] & 0 \qquad (3.2) \\
   \varepsilon_2 \quad 0  \ar[r] & U_0 \ar[r]^{\alpha_0} & U_0\oplus U'_0 \ar[r]^{\beta_0} & U'_0 \ar[r] & 0   }$$
with the exact short sequences $ \varepsilon_2$ and $ \varepsilon_3$ split. In particular, we have a short sequence in $\mod \Lambda$
$$0\to K_1\stackrel{\alpha_2}\to K_2\stackrel{\beta_2}\to K_3 \qquad(3.3)$$

Since the functor $\Hom_{\Lambda}(U,-)$ is left exact and
$D_i=\Coker\Hom_{\Lambda}(U, f_i)$ for all $i=1,2,3$. By applying the functor $\Hom_{\Lambda}(U,-)$ to the commutative exact diagram (3.2) and taking the cokernels,
we have  the following exact commutative diagram
$$\xymatrix{\quad  & 0 \ar[d]  & 0 \ar[d]  &0\ar[d] \\
0 \ar[r] & \Hom_{\Lambda}(U,K_1) \ar[d] \ar[r]^{(U,\alpha_2)} & \Hom_{\Lambda}(U,K_2) \ar[d] \ar[r]^{(U,\beta_3)} & \Hom_{\Lambda}(U,K_3)\ar[d] \\
 \varepsilon_1\quad 0 \ar[r] & \Hom_{\Lambda}(U,U_1) \ar[d]_{(U,f_1)} \ar[r]^{(U,\alpha_1)} & \Hom_{\Lambda}(U,U_1\oplus U'_1) \ar[d]_{(U,f_2)} \ar[r]^{(U,\beta_1)} & \Hom_{\Lambda}(U,U'_1)\ar[d]_{(U,f_3)} \ar[r] & 0 \qquad(3.4) \\
 \varepsilon_0\quad 0 \ar[r] & \Hom_{\Lambda}(U,U_0) \ar[d]\ar[r]^{(U,\alpha_0)} & \Hom_{\Lambda}(U,U_0\oplus U'_0)\ar[d] \ar[r]^{(U,\beta_0)} & \Hom_{\Lambda}(U,U'_1) \ar[d] \ar[r] & 0\\
  & D_1  \ar[d]  & D_2  \ar[d] & D_3  \ar[d]&  \\
    & 0   & 0  & 0 &    }$$

Comparing the commutative exact diagram (3.1) with (3.4), there exists an exact sequence
$$0\to \Hom_{\Lambda}(U,K_1)\stackrel{(U,\alpha_2)}\to \Hom_{\Lambda}(U,K_2)\stackrel{(U,\beta_2)}\to \Hom_{\Lambda}(U,K_3)\to 0\qquad(3.5)$$
by Snake lemma.

On the other hand, noting that the functor $\mathrm{T}_N$ is an exact functor, applying $\mathrm{T}_N$ to the diagram (3.2) gives rise to the following commutative exact diagram in $\mod\Gamma$:
$$\xymatrix{
   & 0 \ar[d]  & 0 \ar[d]  & 0 \ar[d]\\
  0\ar[r]& \mathrm{T}_N(K_1) \ar[d] \ar[r]^{\mathrm{T}_N(\alpha_2)} & \mathrm{T}_N(K_2)\ar[d] \ar[r]^{\mathrm{T}_N(\beta_2)}& \mathrm{T}_N(K_3) \ar[d] \\
  \varepsilon_5 \quad 0  \ar[r] & \mathrm{T}_N(U_1) \ar[d]_{\mathrm{T}_N(f_1)} \ar[r]^{\mathrm{T}_N(\alpha_1)} & \mathrm{T}_N(U_1\oplus U'_1) \ar[d]_{\mathrm{T}_N(f_2)} \ar[r]^{\mathrm{T}_N(\beta_1)} & \mathrm{T}_N(V_1) \ar[d]_{\mathrm{T}_N(f_3)} \ar[r] & 0 \qquad (3.6) \\
   \varepsilon_4 \quad 0  \ar[r] & \mathrm{T}_N(U_0) \ar[r]^{\mathrm{T}_N(\alpha_0)} & \mathrm{T}_N(U_0\oplus U'_0) \ar[r]^{\mathrm{T}_N(\beta_0)} & \mathrm{T}_N(V_0) \ar[r] & 0   }$$
where the short exact sequences  $\varepsilon_4$ and $\varepsilon_5$ are split.

As the functor $\Hom_{\Gamma}(V, -)$ is left exact and
$\mathbb{F}(D_i)$ = $\Coker(\Hom_{\Gamma}(V, \mathrm{T}_N(f_{i}))$ for $i=1,2,3$.
By applying the functor $\Hom_{\Gamma}(V, -)$ to the diagram (3.6)
and taking the cokernels, one  obtains  the following exact commutative diagram in $\mod B$
$$\xymatrix{
   & 0 \ar[d]  & 0 \ar[d] & 0 \ar[d] \\
 0  \ar[r] & \Hom_{\Gamma}(V,\mathrm{T}_N(K_1)) \ar[d]\ar[r]^{(\mathrm{T}_N(\alpha_2))_*} & \Hom_{\Gamma}(V, \mathrm{T}_N(K_2)) \ar[d] \ar[r]^{(\mathrm{T}_N(\beta_2))_*} & \Hom_{\Gamma}(V,\mathrm{T}_N(K_3)) \ar[d]\\
 0 \ar[r]  &\Hom_{\Gamma}(V,\mathrm{T}_N(U_1)) \ar[d]_{(\mathrm{T}_N(f_1))_*} \ar[r]^{(\mathrm{T}_N(\alpha_1))_*} & \Hom_{\Gamma}(V,\mathrm{T}_N(U_1\oplus U'_1)\ar[d]_{(\mathrm{T}_N(f_2))_*} \ar[r]^{(\mathrm{T}_N(\beta_1))_*} &\Hom_{\Gamma}(V,\mathrm{T}_N(U'_1)) \ar[d]_{(\mathrm{T}_N(f_3))_*} \ar[r] & 0 \\
 0 \ar[r]  &\Hom_{\Gamma}(V,\mathrm{T}_N(U_0)) \ar[d] \ar[r]^{(\mathrm{T}_N(\alpha_0))_*} & \Hom_{\Gamma}(V,\mathrm{T}_N(U_0\oplus U'_0)\ar[d]\ar[r]^{(\mathrm{T}_N(\beta_0))_*} &\Hom_{\Gamma}(V,\mathrm{T}_N(V_0)) \ar[d]\ar[r] & 0 \\
    & \mathbb{F}(D_1)\ar[d]  & \mathbb{F}(D_2) \ar[d] & \mathbb{F}(D_3) \ar[d] &   \\
    & 0  & 0  & 0 &  }$$
 where $(-)_*$ stands for $\Hom_{\Gamma}(V, -)$.

We claim that $(\mathrm{T}_N(\beta_2))_*$ in the diagram above is epic.
Indeed, applying the exact functor $\mathrm{T}_N$ to the short exact sequence (3.3) yields an exact sequence in $\mod\Gamma$:
   $$0\to \mathrm{T}_N(K_1)\stackrel{\mathrm{T}_N(\alpha_2)}\to \mathrm{T}_N(K_2)\stackrel{\mathrm{T}_N(\beta_2)}\to \mathrm{T}_N(K_3) \qquad (3.7)$$
Since $(\mathrm{T}_M, \mathrm{T}_N)$ is an adjoint pair and the functors $\Hom_{\Lambda}(\mathrm{T}_M(V),-)$ and $\Hom_{\Gamma}(V, -)$ are left exact,
by applying the functor $\Hom_{\Lambda}(\mathrm{T}_M(V), -)$ to the exact sequence (3.3), and $\Hom_{\Gamma}(V, - )$ to the sequence (3.7),  respectively,
one has the following commutative exact diagram
$$\xymatrix{
 \quad 0 \ar[r] & \Hom_{\Lambda}(\mathrm{T}_M(V), K_1) \ar[d]_{\cong} \ar[r]^{(\alpha_2)_\dagger} & \Hom_{\Lambda}(\mathrm{T}_M(V),K_2) \ar[d]_{\cong} \ar[r]^{(\beta_2)\dagger} & \Hom_{\Lambda}(\mathrm{T}_M(V),K_3)\ar[d]_{\cong}  &  \\
 \quad 0 \ar[r] & \Hom_{\Gamma}(V,\mathrm{T}_N(K_1)) \ar[r]^{(\mathrm{T}_N(\alpha_2))_{\ast}} & \Hom_{\Gamma}(V, \mathrm{T}_N(K_2))\ar[r]^{((\mathrm{T}_N(\beta_2))_{\ast}} & \Hom_{\Gamma}(V, \mathrm{T}_N(K_3))& }$$
where $(-)_{\dagger}$ stands for $\Hom_{\Lambda}(\mathrm{T}_M(V), -)$.
Because of $\mathrm{T}_M(V)\in \add U$ and the exactness of the sequence (3.5),
one has that $(\beta_2)\dagger$ in  the diagram above is epic,
and so is $((\mathrm{T}_N(\beta_2))_{\ast}$, as claimed.


By Snake lemma again, one gets an exact sequence
$0\to  \mathbb{F}(D_1)\to  \mathbb{F}(D_2)\to  \mathbb{F}(D_3)\to 0$ in mod$B$,
which means that the functor $\mathbb{F}$ is exact.

Similarly, we have that the functor $\mathbb{G}$ is exact.

{\bf Step 2}  We verify that $\mathbb{F}$ and $\mathbb{G}$ take projective modules to projective modules.

Let $P_0$ be a projective $A$-module. By Lemma 2.5,  there exists a $\Lambda$-module  $U_0 \in \add_{\Lambda} U$ such that $P_0=\Hom_{\Lambda}(U,U_0)$, which implies $\mathrm{H}_A^{-1}(P_0)=(0,U_0, 0)$. Notice that  $\mathrm{T}_N(U_0)\in \add (\mathrm{T}_N(U))\subseteq \add V$,
 one has $\widetilde{\mathrm{T}_N}\mathrm{H}_A^{-1}(P_0)=(0, \mathrm{T}_N(U_0), 0)$. So that  $\mathbb{F}(P_0)=\Hom_{\Gamma}(V,\mathrm{T}_N(U_0))$
is a projective $B$-module.

The case of $\mathbb{G}$ can be proved in a similar way.

{\bf Step 3} We verity that $\mathbb{F}$ and $\mathbb{G}$ induce a  separable equivalence between $A$ and $B$.

Since $\mathbb{F}$ and $\mathbb{G}$ are additive exact functors,
there exist natural isomorphisms
$\mathbb{F}\cong \mathbb{F}(A)\otimes_A-: \mod A\to \mod B$
and $\mathbb{G}\cong \mathbb{G}(B)\otimes_B-: \mod B\to \mod A$,
where $\mathbb{F}(A)=\Hom_{\Gamma}(_\Gamma V_B, \\ {_\Gamma}N\otimes_{\Lambda}U_A)$ is a $(B,A)$-bimodule
and $\mathbb{G}(B)=\Hom_{\Lambda}(_{\Lambda}U_A,   {_\Lambda}M\otimes_{\Gamma}V_B)$ is an $(A,B)$-bimodule, by Watt\rq s Theorem (see \cite[Theorem 3.33]{R1}).
Noting that $\mathbb{F}$ takes the projective modules to projective modules,
we have $_B\mathbb{F}(A)\cong {_B}\mathbb{F}(A)\otimes_AA$ is a projective $B$-module.
On the other hand, since $\mathbb{F}\cong  {_B}\mathbb{F}(A)\otimes_A-$ is an exact functor,
one gets that $\mathbb{F}(A)_A$ is also a projective right $A$-module.
Similarly, we have that $_A\mathbb{G}(B)_B$ is projective as one-sided module.

Note that $\mathbb{F}$ and $\mathbb{G}$ are additive exact functors,
so is the composition of $\mathbb{F}$ and $\mathbb{G}$.
By Watt\rq s Theorem  again, we have $\mathbb{G}\circ \mathbb{F}\cong \mathbb{GF}(A)\otimes_A-: \mod A\to \mod A$,
where $ \mathbb{GF}(A)=\Hom_\Lambda (_{\Lambda}U_A, {_\Lambda}M\otimes_\Gamma N\otimes_\Lambda U_A)$ is an $A$-bimodule.
On the other hand, there are natural isomorphisms
$\mathbb{G}\circ \mathbb{F}\cong (\mathbb{G}(B)\otimes_B-)\circ( \mathbb{F}(A)\otimes_A-)\cong (\mathbb{G}(B)\otimes_B \mathbb{F}(A))\otimes_A-$.
Thus we have an $A$-bimodule isomorphism
  $_A\mathbb{G}(B)\otimes_B \mathbb{F}(A)_A\cong _A\mathbb{GF}(A)_A=\Hom_{\Lambda}(_{\Lambda}U_A, M\otimes_\Gamma N\otimes_\Lambda U_A)$.

Since $\Lambda$ and $\Gamma$ are symmetrically separably equivalent by assumption, there exists $\Lambda$-bimodule isomorphism
$\rho=(\rho_1, \rho_2): M\otimes_\Gamma N \to \Lambda\oplus X$.
Note that $A=\Hom_{\Lambda}(U_A,U_A)$ and $X^{\prime} =:\Hom_{\Lambda}(_\Lambda U_A, _\Lambda X\otimes_\Lambda U_A)$ are  $A$-bimodules,
it is not hard to check that the natural isomorphism
$\bar{\rho}: \Hom_{\Lambda}(U,  M\otimes_\Gamma N\otimes_\Lambda U_A)\to
\Hom_{\Lambda}(U,U)\oplus \Hom_{\Lambda}(_\Lambda U_A, _\Lambda X\otimes_\Lambda U_A)$
given by $\bar{\rho}(f)=(\mu(\rho_1\otimes \mathrm{Id}_U)f, (\rho_2\otimes \mathrm{Id}_U)f)$,
where $\mu: \Lambda\otimes_{\Lambda}U\to U$ is the multiplication map,
is an $A$-bimodule isomorphism.
Therefore,  $_A\mathbb{G}(B)\otimes_B \mathbb{F}(A)_A\cong A\oplus X^{\prime}$, as desired.

Similarly, one has that there exists a $B$-bimodule isomorphism
$\mathbb{F}(A)\otimes_A \mathbb{G}(B)\cong B\oplus Y'$, where $Y'= \Hom_{\Gamma}(_{\Gamma}V_B,
{_\Gamma}Y\otimes_{\Gamma}V_B)$ is a $B$-module.

{\bf Step 4}  We verity that $(\mathbb{F}(A)\otimes_A-, \mathbb{G}(B)\otimes_B-)$ and $(\mathbb{G}(B)\otimes_B-, \mathbb{F}(A)\otimes_A-)$ are adjoint pairs.

Since $N\otimes_{\Lambda}U\in \add V$ and $B=\Hom_{\Gamma}(V,V)$, we have isomorphisms
\begin{align*}
&\Hom_B(\mathbb{F}(A), B)\\
&=\Hom_B(\Hom_{\Gamma}(V, N\otimes_{\Lambda}U), \Hom_{\Gamma}(V, V))\\
&\cong \Hom_{\Gamma}(N\otimes_{\Lambda}U, V)  \quad (\text{by Yoneda's lemma})\\
&\cong \Hom_{\Lambda}(U, M\otimes_{\Gamma}V)\quad (\text{by the adjoint pair $(N\otimes_{\Lambda}-, M\otimes_{\Gamma}-)$})\\
&=\mathbb{G}(B).
\end{align*}
Thus, for any $B$-module $L$, $\mathbb{G}(B)\otimes_BL\cong \Hom_B(\mathbb{F}(A), B)\otimes_BL\cong \Hom_B(\mathbb{F}(A), L)$,
since $\mathbb{F}(A)$ is a projective $B$-module. This implies that  $(\mathbb{F}(A)\otimes_A-, \mathbb{G}(B)\otimes_B -)$ is an adjoint pair.
In a similar way, we have that $(\mathbb{G}(B)\otimes_B-, \mathbb{F}(A)\otimes_A-)$ is also an adjoint pair.
The proof of Theorem 3.1 is completed.
\end{proof}

As applications of Theorem 3.1, we shall construct new symmetrical separable equivalences from given ones.

Recall that an Artin algebra $\Lambda$ is said to be {\it representation-finite},
provided that there exists finitely many non-isomorphic indecomposable $\Lambda$-modules.
Recall that a $\Lambda$-module $U$ is called an {\it additive generator},
if $\add_{\Lambda}U=\mod \Lambda$. It is clear to see that an Artin algebra is representation-finite if and only of it has an additive
generator. Let $\Lambda$ be a representation-finite Artin algebra. The endomorphism algebra of an additive generator for $\Lambda$
is called {\it the Auslander algebra } of $\Lambda$, which is unique up to Morita equivalence. According to Theorem 3.1, we have the following result.

\begin{theorem}
Let $\Lambda$ and $\Gamma$ be representation-finite. If $\Lambda$ and $\Gamma$ are symmetrically separably equivalent,
then their Auslander algebras are also symmetrically separably equivalent.
\end{theorem}
\begin{proof}
Let $_{\Lambda}U$ and $_{\Gamma}V$ be additive generator for $\Lambda$ and $\Gamma$, respectively. Then $U$ and $V$ satisfy the conditions in
Theorem 3.1, and so the theorem follows.
\end{proof}

Let $\Lambda$ be an Artin algebra. Recall from \cite{HX} that a projective $\Lambda$-module $P$ is said to be {\it $\nu$-stable projective},
if $\nu^i(P)$ is projective for all $i\geq 1$, where $\nu=\mathrm{D}\Hom_{\Lambda}(-,\Lambda)$ is the Nakayama functor of $\Lambda$.
We use $\Lambda$-stp to denote  subcategory of $\mod\Lambda$ consisting of all non-isomorphic $\nu$-stable projective $\Lambda$-modules.
Following \cite{HX}, {\it the Frobenius part} of $\Lambda$, which is self-injective and unique up to Morita equivalences,
is defined to be the endomorphism algebra of  a projective $\Lambda$-module $V$ with $\add_{\Lambda}V= \Lambda$-stp.
And $\Lambda$ is said to be {\it Frobenius-finite}, if its Frobenius part is representation-finite.
Frobenius parts of Artin algebras and Frobenius-finite algebras play crucial roles
in the study of the Auslander-Reiten conjecture, dominant dimensions and stable equivalences (see \cite{CX,HX,LZZ,Xi1} for detail).

\begin{lemma}
Let $\Lambda$ and $\Gamma$ be symmetrically separably equivalent defined by bimodules $_{\Lambda}M_{\Gamma}$ and $_{\Gamma}N_{\Lambda}$,
and let $K$ be a $\Lambda$-module and $L$ a $\Gamma$-module.

$\mathrm{(1)}$ For any $i\geq 1,$  we have
\begin{center}
 $\nu_{\Gamma}^i(N\otimes_{\Lambda}K)\cong N\otimes_{\Lambda}\nu_{\Lambda}^i(K)$, and
 $\nu_{\Lambda}^i(M\otimes_{\Gamma}L)\cong M\otimes_{\Gamma}\nu_{\Gamma}^i(L).$
\end{center}

$\mathrm{(2)}$  If $\Lambda$-stp=$\add_{\Lambda}K$, then we have $\Gamma$-stp=$\add_{\Gamma}(N\otimes_{\Lambda}K);$

$\mathrm{(3)}$  If $\Gamma$-stp=$\add_{\Gamma}L$, then we have $\Lambda$-stp=$\add_{\Lambda}(M\otimes_{\Gamma}L).$
\end{lemma}
\begin{proof}
(1) We only prove $\nu_{\Gamma}^i(N\otimes_{\Lambda}K)\cong N\otimes_{\Lambda}\nu_{\Lambda}^i(K)$.

Note that $(M\otimes_{\Gamma}-, N\otimes_{\Lambda}-)$ and $( N\otimes_{\Lambda}-, M\otimes_{\Gamma}-)$ are adjoint pairs by assumption.
Thus we have
\begin{align*}
\nu_{\Gamma}(N\otimes_{\Lambda}K)
&=\mathrm{D}\Hom_{\Gamma}(N\otimes_{\Lambda}K, \Gamma)\\
&\cong \mathrm{D}\Hom_{\Lambda}(K, M )\\
&\cong \mathrm{D}(\Hom_{\Lambda}(K, \Lambda)\otimes_{\Lambda}M) \quad(\text{because $_{\Lambda}M$ is projective})\\
&\cong \Hom_{\Lambda}(M, \nu_{\Lambda}(K)) \quad(\text{by the adjoint isomorphism} )\\
&\cong \Hom_{\Lambda}(M\otimes_{\Gamma}\Gamma, \nu_{\Lambda}(K))\\
&\cong \Hom_{\Gamma}(\Gamma, N\otimes_{\Lambda}\nu_{\Lambda}(K))\\
& \cong N\otimes_{\Lambda}\nu_{\Lambda}(K).
\end{align*}
Hence $\nu_{\Gamma}^2(N\otimes_{\Lambda}K)\cong \nu_{\Gamma}(\nu_{\Gamma}(N\otimes_{\Lambda}K))\cong
\nu_{\Gamma}(N\otimes \nu_{\Lambda}(K))\cong N\otimes_{\Lambda}\nu_{\Lambda}^2(K).$
The assertion is obtained by induction.

(2) For any $i\geq 1$, we know from (1) that $\nu_{\Gamma}^i(N\otimes_{\Lambda}K)\cong N\otimes_{\Lambda}\nu_{\Lambda}^i(K)$.
Since $K\in\Lambda$-stp, then $\nu_{\Lambda}^i(K)$ is projective,
and so is $N\otimes_{\Lambda}\nu_{\Lambda}^i(K)$ by Lemma 2.4(2),
which yields $N\otimes_{\Lambda}K\in$ $\Gamma$-stp.
For any $L\in$ $\Gamma$-stp, it follows from (1) that $\nu_{\Lambda}^i(M\otimes_{\Gamma}L)\cong M\otimes_{\Gamma}\nu_{\Gamma}^i(L)$
is projective. Thus $M\otimes_{\Gamma}L\in$ $\Lambda$-stp = $\add_{\Lambda}K$,
and hence $N\otimes_{\Lambda}M\otimes_{\Gamma}L\in \add (N\otimes_{\Lambda}K)$.
Notice that $_{\Gamma}L|_{\Gamma}(N\otimes_{\Lambda}M\otimes_{\Gamma}L)$ by Lemma 2.4(3),
we have $\Gamma$-stp = $\add (N\otimes_{\Lambda}K)$, as desired.

(3) The proof is similar to (2).
\end{proof}
\begin{theorem}
Let $\Lambda$ and $\Gamma$ be symmetrically separably equivalent.

$\mathrm{(1)}$ There is a symmetrical separable equivalence between the Frobenius part of $\Lambda$ and that of $\Gamma$;

$\mathrm{(2)}$ $\Lambda$ is Frobenius-finite if and only if so is $\Gamma$.
\end{theorem}
\begin{proof}
Assume that there exist a $\Lambda$-module $K$ and a $\Gamma$-module $L$ such that $\Lambda$-stp =$\add K$ and $\Gamma$-stp =$\add L$, respectively.

(1) From Lemma 3.3(2) and (3), it follows $N\otimes_{\Lambda}K\in$  $\Gamma$-stp = $\add L$ and $M\otimes_{\Gamma}L\in \add K$.
Hence there exists a symmetrical separable equivalence between $\End_{\Lambda}K$ and $\End_{\Gamma}L$ by Theorem 3.1.

(2) Assume  that $\Lambda$ is Frobenius-finite, then $\End_{\Lambda}K$ is  representation-finite.
Noting that $\End_{\Lambda}K$ and $\End_{\Gamma}L$ are symmetrically separably equivalent by (1),
one has that $\End_{\Gamma}L$ is representation-finite  by \cite[Theorem 6]{P1}, which means that $\Gamma$ is Frobenius-finite, as desired.

Analogously, we can prove that $\Lambda$ is Frobenius-finite, when $\Gamma$ is so.
\end{proof}

As a consequence of Theorem 3.4, we get the following corollary,
in which the second result is due to \cite[Corollary 5.4]{HX}.
\begin{corollary}
Let $\Lambda$ be an Artin algebra and $G$ be a finite group such that the order of $G$ is invertible in $\Lambda$.
Suppose that  $\Lambda G$ is the skew group algebra of $\Lambda$ by $G$. Then

$\mathrm{(1)}$ There is a symmetrical separable equivalence between the Frobenius part of $\Lambda G$ and that of $\Lambda$;

$\mathrm{(2)}$ $\Lambda$ is Frobenius-finite if and only if so is $\Lambda G$.
\end{corollary}

\begin{proof} According to  \cite[Lemma 1]{BT}, it follows that $\Lambda$ and $\Lambda G$ are symmetrically separably equivalent.
And this corollary follows from Theorem 3.4 immediately.
\end{proof}

Let $\Lambda$ be an Artin algebra. Recall from \cite{M1} that a finitely generated $\Lambda$-module $T$ is said to be {\it  a $m$-tilting module},
if the following conditions are satisfied. (1) $\pd_{\Lambda} T\leq m$; (2) $\Ext_{\Lambda}^{\geq 1}(T, T)=0$; and (3) $_{\Lambda}\Lambda$ has a $m$-$\add T$-coresolution
$0\to _{\Lambda}\Lambda \to T_0\to T_1\to \cdots \to T_m\to 0.$ Recall from \cite{HR} that an Artin algebra  is called a {\it tilted algebra}, if it is an
endomorphism algebra of a 1-tilting module over a hereditary algebra. It is clear to see that an Artin algebra $\Lambda$ is a tilted algebra
if and only if there is a 1-tilting $\Lambda$-module such that its endomorphism algebra is hereditary. It is well known that tilting modules and tilted algebras are central in the tilting theory.

In the following, we shall construct a new symmetrical separable equivalence from a given
stable equivalence of adjoint type and a tilting module.

\begin{lemma}
Let $\Lambda$ and $\Gamma$ be stably equivalent of adjoint type  defined by $_{\Lambda}M_{\Gamma}$ and $_{\Gamma}N_{\Lambda}$,
and let $T$ be a $m$-tilting $\Lambda$-module. Then

$\mathrm{(1)}$ $N\otimes_{\Lambda}T$ is  a $m$-tilting $\Gamma$-module;

$\mathrm{(2)}$  $M\otimes_{\Gamma}N\otimes_{\Lambda}T\in \add_{\Lambda}T.$
\end{lemma}

\begin{proof} By assumption, there exist projective bimodules $_\Lambda P_{\Lambda}$ and $_{\Gamma}Q_{\Gamma}$, and bimodule isomorphisms
$_\Lambda M\otimes_{\Gamma}N_{\Lambda}\cong \Lambda\oplus P$ and $_\Gamma N\otimes_{\Lambda}M_{\Gamma}\cong \Gamma\oplus Q$
such that $(N\otimes_{\Lambda}-, M\otimes_{\Gamma}-)$ and $(M\otimes_{\Gamma}-, N\otimes_{\Lambda}-)$ are adjoint pairs.

(1) Our proof is based on the definition of a $m$-tilting module.

Assume that $T$ is a $m$-tilting $\Lambda$-module. Note that the exact functor $N\otimes_{\Lambda}-$
takes projective $\Lambda$-modules to projective $\Gamma$-modules by Lemma 2.4(2).
One has $\pd_{\Gamma}(N\otimes_{\Lambda}T)\leq \pd_{\Lambda}T=m$.
On the other hand, since $(M\otimes_{\Gamma}-, N\otimes_{\Lambda}-)$ is an adjoint pair and $M_{\Gamma}$ is projective,
then for any $i\geq 1$,
$\Ext_{\Gamma}^i(N\otimes_{\Lambda}T, N\otimes_{\Lambda}T)\cong \Ext_{\Lambda}^i(M\otimes_{\Gamma}N\otimes_{\Lambda}T, T)
\cong\Ext_{\Lambda}^i((\Lambda\oplus P)\otimes_{\Lambda}T, T)
\cong \Ext_{\Lambda}^i(T\oplus(P\otimes_{\Lambda}T), T)\cong \Ext_{\Lambda}^i(T, T)\oplus \Ext_{\Lambda}^i(P\otimes_{\Lambda}T, T)=0$
from the fact that $ P\otimes_{\Lambda}T$ is a projective $\Lambda$-module.

Since $T$ is a $m$-tilting $\Lambda$-module, there exists a $m$-$\add T $-coresolution:
$$0\to \Lambda \to T_0 \to T_1 \to \cdots \to T_m \to 0.$$
Applying the functor $N\otimes_{\Lambda}-$ to the sequence above gives rise to an exact sequence in $\mod\Gamma$:
$$0\to N \to N\otimes_{\Lambda}T_0 \to N\otimes_{\Lambda}T_1 \to \cdots \to N\otimes_{\Lambda}T_m \to 0 $$
which  is a  proper $m$-$\add (N\otimes_{\Lambda}T)$-coresolution, for the sake of $\Ext_{\Gamma}^{\geq 1}(N\otimes_{\Lambda}T, N\otimes_{\Lambda}T)=0.$
Note that $_{\Gamma}\Gamma |_{\Gamma}N$ by Lemma 2.4(1).  It follows from Lemma 2.6 that there is a proper $m$-$\add(N\otimes_{\Lambda}T)$-coresolution of $_\Gamma\Gamma:$
$0\to \Gamma \to T'_0\to T'_1\to \cdots T'_m\to 0.$
Thus we have proved that $N\otimes_{\Lambda}T$ is a $m$-tilting $\Lambda$-module by the discussion above.

(2) By assumption, there is a $\Lambda$-module isomorphism $M\otimes_{\Gamma}N\otimes_{\Lambda}T \cong T\oplus (P\otimes_{\Lambda}T)$.
To complete our proof, we need to prove $P\otimes_{\Lambda}T\in \add T$. In fact,
since $N\otimes_{\Lambda}T$ is a $m$-tilting module, and $(N\otimes_{\Lambda}-, M\otimes_{\Gamma}-)$ is an adjoint pair, there exist isomorphisms $0=\Ext_{\Gamma}^1(N\otimes_{\Lambda}T, N\otimes_{\Lambda}T)
\cong \Ext_{\Lambda}^1(T, M\otimes_{\Gamma}N\otimes_{\Lambda}T)\cong  \Ext_{\Lambda}^1(T, T\oplus P\otimes_{\Lambda}T)
\cong  \Ext_{\Lambda}^1(T, T)\oplus \Ext_{\Lambda}^1(T, P\otimes_{\Lambda}T)$, which means $\Ext_{\Lambda}^1(T, P\otimes_{\Lambda}T)=0.$
According to \cite[Lemma 3.3]{W}, it follows that $P\otimes_{\Lambda}T \in \gen T$, and hence $P\otimes_{\Lambda}T\in \add T$ by the projectivity of a $\Lambda$-module $P\otimes_{\Lambda}T$, as desired.
\end{proof}

\begin{theorem}
Let $\Lambda$ and $\Gamma$ be stably equivalent of adjoint type  defined by $_{\Lambda}M_{\Gamma}$ and $_{\Gamma}N_{\Lambda}$.
If $T$ is a $m$-tilting $\Lambda$-module, then there exists a symmetrical separable equivalence between  $\End_{\Lambda}T$ and $\End_{\Gamma}(N\otimes_{\Lambda}T)$.
\end{theorem}

\begin{proof}  Since  $T$ is  a $m$-tilting $\Lambda$-module, one has  $M\otimes_{\Gamma}N\otimes_{\Lambda}T\in \add_{\Lambda}T$ by Lemma 3.6.  The assertion follows from Theorem 3.1.

\end{proof}
As a consequence of Theorem 3.7, one  has the following result.

\begin{corollary}
Let $\Lambda$ and $\Gamma$ be stably equivalent of adjoint type.
Then  $\Lambda$ is a tilted algebra if and only if so is $\Gamma.$
\end{corollary}

\begin{proof}
Assume that $\Lambda$ is a tilted algebra. Then there is a 1-tilting $\Lambda$-module $T$ such that $\End_{\Lambda}T$  is a hereditary algebra.
Due to Lemma 3.6, it follows that $N\otimes_{\Lambda}T$  is a 1-tilting $\Gamma$-module.
Because $\End_{\Lambda}T$ and $\End_{\Gamma}(N\otimes_{\Lambda}T)$ are symmetrically separably equivalent by Theorem 3.7,
one gets  $\gd \End_{\Gamma}(N\otimes_{\Lambda}T)=\gd\End_{\Lambda}T \leq 1$ by Lemma 2.4(4),
which implies that $\Gamma$ is a tilted algebra.
\end{proof}

As another application of Theorem 3.1, we obtain the following result, which improves the result of \cite[Theorem 1.1]{L2} a bit.

\begin{corollary}
Let $\Lambda$ and $\Gamma$  be symmetric $k$-algebras and let $M$ be a $(\Gamma, \Lambda)$-bimodule inducing a separable equivalence between
$\Lambda$ and $\Gamma$. Suppose that $U$ is a $\Lambda$-module and $V$ a $\Gamma$-module such that $M^{\ast}\otimes_{\Lambda}U \in \add_{\Gamma}V$
and $M\otimes_{\Gamma}V \in \add_{\Lambda}U$. Then the dominant dimension of $\End_{\Lambda}U$  and of $\End_{\Gamma}V$ are equal.
\end{corollary}

\begin{proof}
According to \cite{K2}, one has that $\Lambda$ and $\Gamma$ are symmetrically separably equivalent induced by
$M$ and its dual module $M^{\ast}=\Hom_{\Lambda}(M, \Lambda)$.
By Theorem 3.1, there is a symmetrical separable equivalence between $\End_{\Lambda}U$ and $\End_{\Gamma}V$.
The assertion follows from the fact that symmetrically separably equivalent algebras have the same dominant dimensions. 
\end{proof}

Let $\Lambda$ be an Artin algebra. Recall from \cite{A} that $V\in\mod \Lambda$ is said to be a {\it Auslander generator}, if $V$ is a generator-cogenerator for $\Lambda$-modules. Recall from \cite{CFK} that  {\it  the rigidity dimension of $\Lambda$}, denoted by rig.dim$\Lambda$, is defined to be
$$\mathrm{rig. dim}(\Lambda)=\sup\{\mathrm{dom. dim}(\End_{\Lambda}(V))| V  \text {is an Auslander generator  in} \mod\Lambda \}.$$
and {\it the representation dimension} of $\Lambda$, denoted by rep.dim$\Lambda$,  is defined to be
$$\mathrm{rep.dim} (\Lambda)=\inf\{\gd(\End_{\Lambda}(V))| V  \text{is an Auslander generator in  } \mod\Lambda \}$$

The notion of rigidity dimension is introduced by Chen {\it et al.} \cite{CFK} to measure homologically the quality of the best resolutions of Artin algebras, and it was proved that the rigidity dimension is an invariant under stable equivalences of adjoint type.
The notion of representation dimension was introduced by Auslander in \cite{A} to measure how far away an Artin algebra is from being representation-finite type in a homological way.

In the following, we shall discuss the invariant properties of the  rigidity dimension and the representation dimension of Artin algebras under symmetrical separable equivalences, and obtain a result about the representation dimension of Xi.
The following lemma is needed.

\begin{lemma}
Let $\Lambda$ and $\Gamma$ be separably equivalent induced by bimodules $_{\Lambda}M_{\Gamma}$ and $_{\Gamma}N_{\Lambda}$, and
let $V$ be a $\Lambda$-module. Suppose that $V$ is an Auslander generator for $\Lambda$-modules, so is  $N\otimes_{\Lambda}V$ as a $\Gamma$-module.
\end{lemma}

\begin{proof}
Since $_{\Lambda}\Lambda \in \add_{\Lambda}V$  and $N$ is a projective generator for $\Gamma$-modules by assumption,
one has $\Gamma\in \add_{\Gamma}(N\otimes_{\Lambda}V)$. On the other hand, let $A$ be a $\Gamma$-module. Then $M\otimes_{\Gamma}A\in \mod\Lambda$, and so there is a $\Lambda$-module monomorphism $f: M\otimes_{\Gamma}A\to V^{(n)}$ for some positive integer $n$,
 because $V$ is a cogenerator for $\Lambda$-modules. Since the functor $N\otimes_{\Lambda}-$ is exact, one gets that  $N\otimes_{\Lambda}f: N\otimes_{\Lambda}M\otimes_{\Gamma}A\to N\otimes_{\Lambda}V^{(n)}$ is a $\Gamma$-module monomorphism. Notice that $_{\Gamma}A|_{\Gamma}(N\otimes_{\Lambda}M\otimes_{\Gamma}A)$ by Lemma 2.4(3), then there exists a
split monomorphism $i:A\to N\otimes_{\Lambda}M\otimes_{\Gamma}A$. So the composition $g=(N\otimes_{\Lambda}f)i: A\to (N\otimes_{\Lambda}V)^{(n)}$ is a $\Gamma$-module monomorphism, which means that $N\otimes_{\Lambda}V$ is a cogenerator for $\Gamma$-modules.
\end{proof}

For a ring $\Lambda$,  let $T$ be a finitely generated $\Lambda$-bimodule. According to Hirata \cite{H2},  $T$ is said to be {\it centrally projective} over $\Lambda$, if $_{\Lambda}T_{\Lambda}\in\add (_{\Lambda}\Lambda_{\Lambda})$.

\begin{theorem} Let $\Lambda$ and $\Gamma$ be symmetrically separably equivalent induced by $_{\Lambda}M_{\Gamma}$ and $_{\Gamma}N_{\Lambda}$.
Assume that $M\otimes_{\Gamma}N$ is centrally projective over $\Lambda$, then

$(\mathrm{1)}$  $\mathrm{rig.dim}(\Lambda)\leq\mathrm{rig.dim}(\Gamma);$

$\mathrm{(2)}$  $\mathrm{rep.dim}(\Lambda)\geq\mathrm{rep.dim}(\Gamma).$
\end{theorem}

\begin{proof}
(1)  If $\mathrm{rig.dim}(\Lambda)=\infty$, there is nothing to prove.

Assume that $\mathrm{rig.dim}(\Lambda)=m<\infty$, then there exists an Auslander generator  $V$  for $\Lambda$-modules  with gl.dim$(\End_{\Lambda}V)<\infty $
 such that $\mathrm{dom.dim}(\End_{\Lambda}V)=m$.  It follows from Lemma 3.10 that  $N\otimes_{\Lambda}V$ is an Auslander generator  for $\Gamma$-modules.

On the other hand, since $_{\Lambda}M\otimes_{\Gamma}N_{\Lambda}\in \add_{\Lambda}\Lambda_{\Lambda}$ by assumption, we have $_{\Lambda}M\otimes_{\Gamma}N\otimes_{\Lambda}V\in \add _{\Lambda}V$.
According to Theorem 3.1,  it follows that $\End_{\Lambda}V$ and $\End_{\Gamma}(N\otimes_{\Lambda}V)$ are symmetrically separably equivalent.
Hence by Lemma 2.4(4) one gets that gl.dim$\End_{\Gamma}(N\otimes_{\Lambda}V)$= gl.dim$\End_{\Lambda}V<\infty$
and $\mathrm{dom.dim}\End_{\Gamma}(N\otimes_{\Lambda}V)=\mathrm{dom.dim}\End_{\Lambda}V=m$. This infers  $\mathrm{rig.dim}(\Lambda)\leq \mathrm{rig.dim}(\Gamma).$

(2) Assume that rep.dim$\Lambda=m$. Then there exists an Auslander generator $V$ for $\Lambda$-modules such that gl.dim$\End_{\Lambda}V=m$.
Lemma 3.10 yields that $N\otimes_{\Lambda}V$ is an Auslander generator. Note that $_{\Gamma}M\otimes_{\Lambda}N_{\Gamma}$ is centrally projective by assumption, it follows from Theorem 3.1 that there  exists a symmetrical separable equivalence between $\End_{\Lambda}V$ and $\End_{\Gamma}(N\otimes_{\Lambda}V)$.
Hence gl.dim$\End_{\Gamma}(N\otimes_{\Lambda}V)=m$ by Lemma 2.4(4), which implies rep.dim$\Gamma\leq m$, as desired.
\end{proof}

\begin{corollary} Let $\Lambda$ and $\Gamma$ be symmetrically separably equivalent induced by bimodules
$_{\Gamma}M_{\Lambda}$ and $_{\Lambda}N_{\Gamma}$.
Suppose that $M\otimes_{\Gamma}N$ and $N\otimes_{\Lambda}M$ are centrally projective over $\Lambda$ and $\Gamma$, respectively, then we have

$\mathrm{(1)}$   $\mathrm{rig.dim}(\Lambda)=\mathrm{rig.dim}(\Gamma).$

 $\mathrm{(2)}$ $\mathrm{rep.dim}  (\Gamma) = \mathrm{rep.dim} (\Lambda).$
\end{corollary}

There are examples of separable equivalences satisfies all conditions of Corollary 3.12.

(1) Let $\Lambda$ be a finite dimensional algebra over an algebraically closed field $k$ and $F$ is a separable field extension of $k$. Following \cite[Lemma 4.6]{SZ}, $\Lambda\otimes_{k}F$
and $\Lambda$ are symmetrically separably equivalent induced by $_{\Lambda\otimes_{k}F}(\Lambda\otimes_{k}F)_{\Lambda}$ and
$_{\Lambda}(\Lambda\otimes_{k}F)_{\Lambda\otimes_{k}F}$. Also  $_{\Lambda}(\Lambda\otimes_{k}F)\otimes_{\Lambda\otimes_{k}F}(\Lambda\otimes_{k}F)$
is centrally projective over $\Lambda$, and $(\Lambda\otimes_{k}F)\otimes_{\Lambda}(\Lambda\otimes_{k}F)$ is centrally projective over $\Lambda\otimes_{k}F$.

(2) Let $\Lambda$ be an Artin $R$-algebra, where $R$ is a commutative Artin ring. If $\Lambda$ is an excellent extension of $R$, it is not difficult to see that $\Lambda$ and $R$ are symmetrically separably equivalent induced by $_\Lambda\Lambda_R$ and $_R\Lambda_\Lambda$. Moreover, $\Lambda\otimes_R\Lambda$ and $\Lambda\otimes_{\Lambda}\Lambda$ are centrally projective over $\Lambda$ and $R$, respectively.

\section{Auslander-type condition and Homological conjectures}\label{sec4}

In this section, we shall prove that the Auslander-type condition and homological conjectures, including the Gorenstein symmetric conjecture, the (strong) Nakayama conjecture and the Auslander Gorenstein  conjecture, are preserved under symmetrical separable equivalences. We begin with the following observation.

\begin{lemma}
Let $\Lambda$ and $\Gamma$ be two rings and let  $I$ be  a $\Lambda$-module and $J$ a right $\Lambda$-module. Suppose that there is a
symmetrical separable equivalence between $\Lambda$ and $\Gamma$ defined by bimodule $_{\Lambda}M_{\Gamma}$ and $_\Gamma N_{\Lambda}$. Then

$\mathrm{(1)}$ $I$ is an injective $\Lambda$-module if and only if so is $N\otimes_{\Lambda}I$ as a $\Gamma$-module.

$\mathrm{(2)}$ $J$ is an injective right $\Lambda$-module if and only if so is $J\otimes_{\Lambda}M$ as a right  $\Gamma$-module.

$\mathrm{(3)}$ $F$ is a flat $\Lambda$-module if and only if so is $N\otimes_{\Lambda}F$ as a $\Gamma$-module.
\end{lemma}

\begin{proof}
(1) Since $M_{\Gamma}$ is projective and $(M\otimes_{\Gamma}-, N\otimes_{\Lambda}-)$ is an adjoint pair,
then for any $i\geq 1$ and any $\Gamma$-module $K$, there is an isomorphism $\Ext_{\Gamma}^i(K, N\otimes_{\Lambda}I)\cong \Ext_{\Lambda}^i(M\otimes_{\Gamma}K, I)=0$,
 which yields that $N\otimes_{\Lambda}I$ is an injective $\Gamma$-module, as desired.

 Conversely, assume that $N\otimes_{\Lambda}I$ is an injective $\Gamma$-module, so is a $\Lambda$-module $M\otimes_{\Gamma}N\otimes_{\Lambda}I$. Notice that $I|M\otimes_{\Gamma}N\otimes_{\Lambda}I$, one gets the assertion.

(2)  The proof is similar to (1).

(3) Let $0\to A\to B\to C\to 0$ be an exact sequence of right $\Gamma$-module.
Since $-\otimes_{\Gamma}N$ is an exact functor and $F$ is flat, then
$0\to A\otimes_{\Gamma}N\otimes_{\Lambda}F\to B\otimes_{\Gamma}N\otimes_{\Lambda}F\to C\otimes_{\Gamma}N\otimes_{\Lambda}F\to 0$ is exact, which implies that $N\otimes_{\Lambda}F$ is flat.

 Conversely, assume that $N\otimes_{\Lambda}F$ is flat, so is $M\otimes_{\Gamma}N\otimes_{\Lambda}F$
 as $\Lambda$-module. Note that $F|M\otimes_{\Gamma}N\otimes_{\Lambda}F$. Hence $F$ is flat.
\end{proof}

\begin{corollary}
Let $\Lambda$ and $\Gamma$ be symmetrically separably equivalent induced by bimodules $_\Lambda M_\Gamma$ and $_\Gamma N_\Lambda$, and let $V$ be a $\Lambda$-module and $W$ a right $\Lambda$-module. Then we have

$\mathrm{(1)}$  $\id_{\Lambda}V=\id_{\Gamma}(N\otimes_{\Lambda}V);$

$\mathrm{(2)}$ $ \id W_{\Lambda}=\id (W\otimes_{\Lambda}M)_{\Gamma};$

$\mathrm{(3)}$  $\fd_{\Lambda}V=\fd_{\Gamma}(N\otimes_{\Lambda}V);$

$\mathrm{(4)}$  $\fd W_{\Lambda}=\fd(W\otimes_{\Lambda}M)_{\Gamma};$

$\mathrm{(5)}$ $\pd_{\Lambda}V=\pd_{\Gamma}(N\otimes_{\Lambda}V);$

$\mathrm{(6)}$  $\pd W_{\Lambda}=\pd_{\Gamma}(W\otimes_{\Lambda}M).$
\end{corollary}

\begin{proof}
We only prove (1), the proofs of $(2)-(6)$ are similar.

Without loss of generality, we may consider $\id_{\Lambda}V=n <\infty$. Then there is an exact sequence in $\Mod\Lambda$:
$0\to _\Lambda V\to E_0\to E_1\to \cdots \to E_n\to 0,$ with each $E_i$ injective.
Since $N$  is a projective right $\Lambda$-module,
 we get a new exact sequence
$$0\to N\otimes_\Lambda V \to N\otimes_{\Lambda} E_0\to N\otimes_{\Lambda} E_1\to \cdots \to N\otimes_{\Lambda} E_n\to 0$$
where  $N\otimes_{\Lambda} E_i$ is an  injective $\Gamma$-module for each $0\leq i\leq n$, by Lemma 4.1.
This means $\id_{\Gamma}(N\otimes_\Lambda V)\leq n.$

Conversely, assume that  $\id_{\Gamma}(N\otimes_\Lambda V) = m,$ then  $\id_{\Lambda}(M\otimes_{\Gamma}N\otimes_\Lambda V)\leq m.$
Notice that $_{\Lambda}V|_{\Lambda}(M\otimes_{\Gamma}N\otimes_\Lambda V)$ by Lemma 2.4(3),
it follows that $\id _{\Lambda}V \leq m$. The assertion follows immediately.
\end{proof}

\begin{corollary}
Let $\Lambda$ and $\Gamma$ be symmetrically separably equivalent. Then

$\mathrm{(1)}$ ~~$\id_{\Lambda}\Lambda = \id_{\Gamma}\Gamma;$

$\mathrm{(2)}$ ~~$\id\Lambda_{\Lambda} = \id\Gamma_{\Gamma}.$
\end{corollary}

\begin{proof}
 (1) According to Corollary 4.2(1), it follows that $\id_{\Lambda}\Lambda=\id_{\Gamma}N.$
Combining a fact that $_{\Gamma}N$ is a projective generator, one gets $\id_{\Lambda}\Lambda=\id_{\Gamma}\Gamma$.

 (2) is dual to (1).
\end{proof}

\begin{lemma}
Let $$0\to {_\Lambda \Lambda} \to I^{0}(\Lambda)\to I^{1}(\Lambda)\to \cdots \to I^{n}(\Lambda)\to \cdots \qquad (4.1)$$
and
$$0\to {_\Gamma\Gamma} \to I^{0}(\Gamma)\to I^{1}(\Gamma)\to \cdots \to I^{n}(\Gamma)\to \cdots  \qquad (4.2)$$
be the minimal injective resolution of $ _\Lambda \Lambda$ and of $_\Gamma\Gamma$, respectively.
For any $i\geq 0$, we have

$\mathrm{(1)}$ ~~$\fd_{\Lambda}I^{i}(\Lambda) = \fd_{\Gamma}I^{i}(\Gamma);$

$\mathrm{(2)}$ ~~$\pd_{\Lambda}I^{i}(\Lambda) = \pd_{\Lambda}I^{i}(\Gamma).$
\end{lemma}

\begin{proof} We only prove (1), and the proof of (2) is similar.

Applying the exact functor $ N\otimes_{\Lambda}-$ to the injective resolution  (4.1) induces an exact sequence
$$0\to _\Gamma N \to N\otimes_{\Lambda}I^0(\Lambda)\to N\otimes_{\Lambda}I^1(\Lambda)\to \cdots \to N\otimes_{\Lambda}I^n(\Lambda)\to \cdots  $$
where $N\otimes_{\Lambda}I^i(\Lambda)$ is an injective $\Gamma$-module  for all $i\geq 0$ by Lemma 4.1.
Note that the sequence (4.2) is the minimal injective resolution of $_{\Gamma}\Gamma$ and $N$ is a projective generator for $\Gamma$-modules  by Lemma 2.4(1), it is not hard to check that $I^{i}(\Gamma)|N\otimes_{\Lambda}I^{i}(\Lambda) $ for all $i\geq 0$.
Hence, for each $i$, we have $\fd_{\Lambda}I^{i}(\Gamma)\leq \fd_{\Gamma}(N\otimes_{\Lambda}I^{i}(\Lambda)) =\fd_{\Lambda}I^{i}(\Lambda)$ by Corollary 4.2.

Dually, we have  $\fd_{\Lambda}I^{i}(\Lambda)\leq\fd_{\Gamma}I^{i}(\Gamma)$ for any $i\geq 0$.
\end{proof}

\begin{theorem} Let $\Lambda$ and $\Gamma$ be symmetrically separably equivalent and let $n, k$ be integers. Then
 $\Lambda$ is $G_n(k)$ if and only if so is $\Gamma$.
\end{theorem}

\begin{proof}
Suppose that $\Lambda$ satisfies $G_n(k)$ , then we have $\fd I^i(\Lambda) \leq i+k$ for any $0\leq i\leq n-1$.
It follows from Lemma 4.4 that $\fd I^i(\Gamma) \leq i+k$ for any $0\leq i\leq n-1$. Thus, we verify that $\Gamma$ satisfies $G_n(k)$.

Similarly, it can be verify that  $\Lambda$ is $G_n(k)$ when $\Gamma$ is so.
\end{proof}

Let $n$ be a positive integer. Following \cite{Hu1}, a left and right noetherian ring $\Lambda$ is {\it left quasi $n$-Gorenstein } if
the left flat dimension of the $i$th term in the minimal injective resolution of $\Lambda$ as a left $\Lambda$-module is less than or equal to
$i+1$ for any $0\leq i\leq n-1$, which is a nontrivial generalization of $n$-Gorenstein rings. It is easy to see that a left and right
noetherian ring is a left quasi $n$-Gorenstein if and only if it is $G_{n}(1)$. The left quasi $n$-Gorenstein ring plays an important role in the
studying of the Auslander Gorenstein conjecture and the Nakayama conjecture (see \cite{Hu1,Hu2} for detail).

 According to Theorem 4.5, we have the following corollary.

\begin{corollary} Let $\Lambda$ and $\Gamma$ be symmetrically separably equivalent,
 and let $n$ be a nonnegative integer. Then

$\mathrm{(1)}$ $\Lambda$ is $n$-Gorenstein if and only if so is $\Gamma$;

$\mathrm{(2)}$ $\Lambda$ is quasi $n$-Gorenstein if and only if so is $\Gamma$;

$\mathrm{(3)}$ $\Lambda$ satisfies the Auslander condition if and only if so does $\Gamma$.
\end{corollary}

\begin{proof}
Note that $\Lambda$ is $n$-Gorenstein if and only if $\Lambda$ is $G_n(0)$, and that $\Lambda$ is quasi $n$-Gorenstein if and only if $\Lambda$ is $G_n(1)$. Thus (1) and (2)  are obtained from Theorem 4.5 directly.
(3) follows from (1) directly.
\end{proof}

Finally, we shall show that the symmetrical separable equivalence preserves the Gorenstein symmetric conjecture, the (strong) Nakayama conjecture and the Auslander Gorenstein conjecture.

\begin{theorem} Let $\Lambda$ and $\Gamma$ be Artin $R$-algebras.

 $\mathrm{(1)}$ $\Lambda$ satisfies the Gorenstein symmetric conjecture if and only if so does $\Gamma$;

$\mathrm{(2)}$ $\Lambda$ satisfies the strong Nakayama conjecture if and only if so does $\Gamma$.

$\mathrm{(3)}$  $\Lambda$ satisfies the Auslander-Gorenstein conjecture  if and only if so does $\Gamma$.

$\mathrm{(4)}$  $\Lambda$ satisfies the Nakayama conjecture  if and only if so does $\Gamma$.
\end{theorem}

\begin{proof}
(1) According to Corollary 4.3, it follows that $\id_{\Lambda}\Lambda=\id\Lambda_{\Lambda}$ if and only if $\id_{\Gamma}\Gamma=\id\Gamma_{\Gamma}$, and the assertion follows.

(2) Assume that $\Gamma$ satisfies the strong Nakayama conjecture. Let $V$ be a $\Lambda$-module with $\Ext_{\Lambda}^{\geq 0}(V,\Lambda)=0$.
It is not hard to check that $\Ext_{\Lambda}^{\geq 0}(V, P)=0$ for any finitely generated projective $\Lambda$-module $P$.
Since $N_\Lambda$ is projective and $(N\otimes_{\Lambda}-, M\otimes_{\Gamma}-)$ is an adjoint pair,
then for any $i\geq 0$, $\Ext_{\Gamma}^i(N\otimes_{\Lambda}V, \Gamma)\cong \Ext_{\Lambda}^i(V, M)=0$.
Therefore, one has $N\otimes_\Lambda V=0$ by assumption, which implies $V=0$, because $N_{\Lambda}$ is a projective generator for right $\Lambda$-modules.
Thus $\Lambda$ satisfies the strong Nakayama conjecture.

Similarly, we can prove that $\Gamma$ satisfies the strong Nakayama conjecture, when $\Lambda$ does so.

(3) Assume that $\Gamma$ satisfies the Auslander-Gorenstein conjecture and that $\Lambda$ satisfies the Auslander condition. Due to Corollary 4.6(3), it follows that $\Gamma$ satisfies the Auslander condition, and hence one has that $\Gamma$ is a Gorenstein algebra ({\it that is,}
$\id_{\Gamma}\Gamma <\infty$ and $\id\Gamma_{\Gamma}<\infty$ ) by assumption,  which implies  that $\Lambda$ is a Gorenstein algebra by Corollary 4.3.

Similarly, we can prove that $\Gamma$ satisfies the Auslander-Gorenstein conjecture, when $\Lambda$ does so.

(4) Assume that $\Gamma$ satisfies the Nakayama conjecture and dom.dim$\Lambda=\infty$. Then one gets dom.dim$\Gamma=\infty$ by Lemma 4.4(2).
 so that $\Gamma$ is self-injective by assumption. It follows from Corollary 4.3
that $\Lambda$ is self-injective, as desired.

Similarly, we can prove that $\Gamma$ satisfies the Nakayama  conjecture, when $\Lambda$ does so.
\end{proof}

As a consequence of Theorem 4.7, we obtain the following corollary, which improves \cite[Theorem 3.3 (2)]{Z}.

 \begin{corollary}
 Let $\Lambda$ be an Artin algebra, and let $G$ be a finite group acting on $\Lambda$, whose order is invertible in $\Lambda$.
Then $\Lambda$ satisfies the strong Nakayama conjecture if and only if so does the skew group $\Lambda G$.
 \end{corollary}


\section{Declarations}
 
\bmhead{Ethical Approval} 

Not applicable.
 
\bmhead{Funding} 

This research was partially supported by the National Natural Science Foundation of China (Grant No. 12061026) and Foundation for University Key Teacher by Henan Province (2019GGJS204).
 
\bmhead{Availability of data and materials} 

Not applicable.
 


\end{document}